\documentclass[a4paper,12pt]{amsart}
\usepackage{amsfonts}
\usepackage{amsthm}
\usepackage{amssymb}
\usepackage{amsmath}
\usepackage{theoremref}
\usepackage{enumerate}
\usepackage{bbm}
\usepackage{bm}
\usepackage{pgfplots}
\pgfplotsset{compat=1.18} 
\usepgfplotslibrary{fillbetween}
\usepackage{mathtools}
\usepackage{hyperref}

\numberwithin{equation}{section}

\newcommand{\T}{\mathbb{T}}

\newcommand{\conj}[1]{\overline{#1}}
\newcommand{\D}{\mathbb{D}}

\newcommand{\R}{\mathbb{R}}

\newcommand{\cD}{\conj{\mathbb{D}}}
\newcommand{\ip}[2]{\big\langle #1, #2 \big\rangle}

\newcommand{\m}{\textit{m}}

\newcommand{\hb}{\mathcal{H}(b)}

\newcommand{\h}{\mathcal{H}}

\renewcommand\Re{\operatorname{Re}}

\newcommand{\supp}[1]{\text{supp}({#1})}

\newtheorem{mainthm}{Theorem}

\newtheorem{thm}{Theorem}[section]
\newtheorem*{thm*}{Theorem}
\newtheorem{lem}[thm]{Lemma}

\newtheorem{cor}[thm]{Corollary}
\newtheorem*{cor*}{Corollary}

\theoremstyle{definition}

\theoremstyle{definition}
\newtheorem{remark}[thm]{Remark}

\newtheorem*{defn*}{Definition}

\newtheorem*{question*}{Question}

\newtheorem{claim*}{Claim}

\let\oldtocsection=\tocsection
\let\oldtocsubsection=\tocsubsection
\let\oldtocsubsubsection=\tocsubsubsection
\renewcommand{\tocsection}[2]{\hspace{0em}\oldtocsection{#1}{#2}}
\renewcommand{\tocsubsection}[2]{\hspace{1em}\oldtocsubsection{#1}{#2}}
\renewcommand{\tocsubsubsection}[2]{\hspace{2em}\oldtocsubsubsection{#1}{#2}}
\def\section{\@startsection{section}{1}%
\z@{.7\linespacing\@plus\linespacing}{.5\linespacing}%
{\large\scshape\centering}}

\title{Tangential boundary behavior in Hilbert spaces of analytic functions}
\author{Shuaibing Luo}
\address{School of Mathematics, Hunan University, Changsha, 410082, PR China}
\email{sluo@hnu.edu.cn}

\author{Bartosz Malman}
\address{Division of Mathematics and Physics, M\"{a}lardalen University, 721 23 V\"{a}ster{\aa}s, Sweden}
\email{bartosz.malman@mdu.se}

\begin{document}

\begin{abstract}
Sarason's Hilbert space version of Carathéodory-Julia Theorem connects the non-tangential boundary behavior of functions in de Branges-Rovnyak space $\hb$ with the existence of angular derivatives in the sense of Carathéodory for $b$, an analytic self-mapping of the unit disk. In this article, we continue the study of higher order extensions of this result that deal with derivatives of functions in $\hb$, and we consider notions of approach regions more general than the non-tangential ones. Our main result generalizes the recent work of Duan-Li-Mashreghi on boundary behavior in model spaces to $\hb$-spaces and to higher order derivatives, and we give a new self-contained proof of that result. It also generalizes earlier radial results of Fricain-Mashreghi. In relation to existence of angular derivatives, we show that in the classical Carathéodory-Julia Theorem one cannot replace the non-tangential approach region by any essentially larger region.
\end{abstract}

\thanks{Luo is supported by NNSFC (12271149), Natural Science
Foundation of Hunan Province (2024JJ2008).  Malman is supported by Vetenskapsrådet (VR2024-03959).}

\maketitle

\tableofcontents

\section{Background}

\subsection{Angular derivatives in the sense of Carathéodory.}

Let $\D := \{ z \in \mathbb{C} : |z| < 1\}$ denote the unit disk. We single out the point $1$ on its boundary circle $\T := \{ z \in \mathbb{C} : |z| = 1\}$ and consider inside $\D$ the ordinary non-tangential approach region $\Gamma$ with vertex at $1$:
\begin{equation}
    \label{E:GammaRegionDef}
    \Gamma = \Gamma(c) = \{ z \in \D : 1-|z| > c|\arg z| \}, \quad c > 0, \, \arg z \in (-\pi, \pi]
\end{equation} (parameter $c$ being fixed but arbitrary). Our choice of $1 \in \T$ as our base point is done for notational convenience. It can be replaced by any other point $\zeta_o \in \T$. The \textit{angular derivative at the point  $1$} of an analytic function $f$ defined in $\D$ is the non-tangential limit \[ f'(1) := \lim_{\substack{z \to 1 \\ z \in \Gamma}} f'(z),\] if it exists. It is widely known that the existence of the limit does not depend on the parameter $c$ used to define $\Gamma$. 

A specialized notion of angular derivative exists in the context of analytic self-mappings of $\D$. We say that $b: \D \to \D$ has an angular derivative \textit{in the sense of Carathéodory at the point $1$} if $b$ has an angular derivative at 1, and additionally $b(1) := \lim_{z \to 1, z \in \Gamma} b(z)$ is unimodular, i.e., $|b(1)| = 1$. The famous Carathéodory-Julia Theorem from \cite{caratheodory1929winkelderivierten} and \cite{julia1920extension} reads as follows.

\begin{thm}{\textbf{(Carathéodory-Julia)}} \thlabel{T:CJTheorem}
Let $b: \D \to \D$ be analytic. The following two conditions are equivalent.
\begin{enumerate}[(i)]
\item The function $b$ has an angular derivative in the sense of Carathéodory at the point $1$.
\item There exists a sequence $(z_n)_n$ of points in $\D$ which converges to $1$ and satisfies 
\begin{equation}
    \label{E:KernelNormBoundSequenceTo1}
    \sup_n \frac{1-|b(z_n)|^2}{1-|z_n|^2} < \infty.
\end{equation} 
\end{enumerate}
\end{thm}
It is known that \eqref{E:KernelNormBoundSequenceTo1} is in fact equivalent to a similar condition
\[ \sup_{z \in \Gamma} \frac{1-|b(z)|^2}{1-|z|} < \infty.\]

The implication $(i) \Rightarrow (ii)$ is immediately seen to hold by setting $z = r \in \R \cap \D$ and noting that $1-|b(r)|$ is dominated by $|b(1) - b(r)|$, a quantity which is $\mathcal{O}(1-r)$ under $(i)$. 
An insightful proof of implication $(ii) \Rightarrow (i)$ has been given by Sarason in \cite{sarason1988angular}, who identified the quantity in $(ii)$ as the norm of the functional of evaluation in a certain Hilbert space of analytic functions, and in effect connected the geometric notion of angular derivatives with operator theory. The de Branges-Rovnyak space $\hb$ is the Hilbert space of analytic functions on $\D$ characterized by the property that it contains the \textit{kernel functions} 
\begin{equation}
    \label{E:HbKernelEq}
    k^b_z(w) := \frac{1-\conj{b(z)}b(w)}{1-\conj{z}w}, \quad z, w \in \D
\end{equation} and its inner product $\ip{\cdot}{\cdot}_b$ satisfies 
\begin{equation}
    \label{E:ReproducingPropertyHB}
    f(z) = \ip{f}{k^b_z}_b, \quad f \in \hb.
\end{equation} 
Good references for background on $\hb$-spaces is Sarason's book \cite{sarasonbook} and the more recent extensive two-volume monograph by Fricain and Mashreghi \cite{hbspaces1fricainmashreghi}, \cite{hbspaces2fricainmashreghi}.

It follows from the reproducing property in \eqref{E:ReproducingPropertyHB} that
\begin{equation}
    \label{E:HbNormEq} 
    \|k^b_z\|_b^2 = \ip{k^b_z}{k^b_z}_b = \frac{1-|b(z)|^2}{1-|z|^2}. 
\end{equation} Under condition $(ii)$ of \thref{T:CJTheorem}, the \textit{boundary kernel function}
\begin{equation}
    \label{E:BoundaryKernelEq}
    k^b_1(w) := \frac{1-\conj{b(1)}b(w)}{1-w}, \quad w \in \D
\end{equation} can be identified as a weak limit of a subsequence of the $(k^b_{z_n})_n$, and is thus a member of $\hb$. Naturally, it should have the evaluation property that $\ip{f}{k^b_1}_b = f(1)$. Indeed this is so, if $f(1)$ is duly interpreted. The following result appears in \cite{sarason1988angular} and is reproduced in Sarason's book \cite{sarasonbook}.

\begin{thm}{\textbf{(Sarason's extension)}}
The following condition is equivalent to the two conditions in the Carathéodory-Julia Theorem.
\begin{enumerate}
    \item[(iii)] We have norm convergence of $k^b_z \to k^b_1$ in $\hb$ as $z \to 1$ in $\Gamma$. In particular, for every $f \in \hb$, the non-tangential limit
    \[ f(1) := \lim_{\substack{z \to 1\\ z \in \Gamma}} f(z)\] exists and satisfies $\ip{f}{k^b_1}_b = f(1)$.
\end{enumerate}
\end{thm}

% Sarason's contribution not only extends the Carathéodory-Julia Theorem by a new equivalent condition, but it also provides an operator-theoretic way to think about angular derivatives. Consequently, quick proofs exist of the implication $(ii) \Rightarrow (i)$, passing by $(iii)$ (one is available in \cite[Chapter VI]{sarasonbook}). 

\subsection{Other equivalent conditions and higher order analogs.}
Consider the Nevanlinna factorization of $b$ as \begin{equation}
    \label{E:bInnerOuterFact}
    b(z) = B(z) S_\nu(z) b_o(z),
\end{equation} where $B$ is a Blaschke product with zero sequence $(a_n)_n$,
\[ B(z) = \prod_n \frac{|a_n|}{a_n}\frac{a_n-z}{1-\conj{a_n}z},\]
$S_\nu$ is the singular inner factor
\[ S_{\nu}(z) = \exp \Big(-\int_\T \frac{\zeta + z}{\zeta - z} d\nu(\zeta) \Big),\]
and $b_o$ is the outer factor
\[b_o(z) = \exp \Big(\int_\T \frac{\zeta + z}{\zeta - z} \log|b(\zeta)| \, d\m(\zeta) \Big).\]
Here $d\m$ is the Lebesgue measure on $\T$, normalized by the condition $m(\T) = 1$. Existence of the angular derivative has been characterized in terms of the quantities appearing in this factorization. We proceed to present these results in the more general context of higher order Carathéodory-Julia theory, which arises in the investigation of conditions for existence of boundary values for higher order derivatives of $b$ and functions in $\hb$.

Let $\partial_z$ and $\conj{\partial_z}$ denote the usual complex differentiation operators with respect to the complex variable $z$. For every $z \in \D$, the limit 
\[ \conj{\partial_z} k^b_z = \lim_{h \to 0} \frac{k^b_{z+h} - k^b_z}{\conj{h}}\] exists in the norm of $\hb$. Similarly, the higher order kernels $\conj{\partial_z}^m k^b_z$ are also members of $\hb$, and it is not hard to see that \begin{equation}
    \label{E:DerivativReproducingFormula} 
    \ip{f}{\conj{\partial_z}^mk^b_z}_b = f^{(m)}(z). 
\end{equation} According to \eqref{E:HbNormEq}, we have the useful formulas
\begin{equation}
    \label{E:HigherOrderKernelNorm}
    \|\conj{\partial_z}^m k^b_z\|_b^2 = \ip{\conj{\partial_z}^m k_z^b}{\conj{\partial_z}^m k_z^b}_b = \partial_z^m \conj{\partial_z}^m \frac{1-|b(z)|^2}{1-|z|^2}.
\end{equation} 
If the quantity in \eqref{E:HigherOrderKernelNorm} is uniformly bounded over a sequence of points $(z_n)_n$ tending to $1$, then we may as before define the \textit{boundary derivative kernel} $\conj{\partial_z}^m k^b_1$ as a weak (and pointwise in $\D$) limit of a subsequence $(\conj{\partial_z^m} k^b_{z_n})_n$. Throughout the article, $\conj{\partial_z}^m k^b_1$ is to be interpreted according to that weak limit definition. We will not need any more precise pointwise formula for this function (but see \cite{fricain2008integral}). 

The following result of Fricain and Mashreghi has been obtained in \cite{fricain2008boundary} and \cite{fricain2008integral} and characterizes the existence of radial limits for derivatives of functions in $\hb$ in terms of the quantities appearing in \eqref{E:bInnerOuterFact}.

\begin{thm}\textbf{(Fricain-Mashreghi, higher order radial characterization)} \thlabel{T:FricainMashreghiTheorem}
For every non-negative integer $m$, the following conditions are equivalent.
\begin{enumerate}[(i)]
    \item As $r \to 1$, we have the norm convergence $\conj{\partial_z}^m k^b_{r} \to \conj{\partial_z}^m k^b_1$ in $\hb$. 
    \item For every $f \in \hb$, the radial limit
    \[ f^{(m)}(1) := \lim_{r \to 1} f^{(m)}(r)\] exists.
    \item We have \[ \sum_n \frac{1-|a_n|^2}{|1-a_n|^{2m+2}} + \int_\T \frac{d\nu(\zeta)}{|1 - \zeta|^{2m+2}} + \int_\T \frac{-\log|b(\zeta)|}{|1 -\zeta|^{2m+2}} \, d\m(\zeta) < \infty.\]
\end{enumerate}
\end{thm}

Earlier results in this direction which pertain to inner $b$ are due to Ahern and Clark and can be found in \cite{ahern1970radial}. It may be worth mentioning that, with some work, one can replace the radial convergence in $(i)$ and $(ii)$ by non-tangential convergence. There exist results in the spirit of \thref{T:FricainMashreghiTheorem} with part $(iii)$ replaced by similar conditions on the Aleksandrov-Clark measure corresponding to $b$ (see the discussion in \cite[Chapter VII]{sarasonbook}). We note also the automatic strengthening from weak convergence $\conj{\partial_z}^m k^b_{r} \to \conj{\partial_z}^m k^b_1$ implicit in $(ii)$ to the norm convergence in $(i)$. 

% Moreover, there exist characterizations expressed in terms of the family of Aleksandrov-Clark measures $\{\mu_\alpha\}_{\alpha \in \T}$ of $b$, namely the measures in the Poisson representations
% \begin{equation}
%     \label{E:AleksandrovClark}
%     \frac{1-|b(z)|^2}{|1-\conj{\alpha}b(z)|^2} = \int_\T \frac{1-|z|^2}{|\zeta - z|^2} \, d\mu_\alpha(\zeta).
% \end{equation}

% Sarason outlines in \cite[Chapter VII]{sarasonbook} the following  higher order generalization of the Carathéodory-Julia Theorem.

% \begin{thm}{\textbf{(Sarason, higher order characterization in terms of Aleksandrov-Clark measures)}}
% The following conditions are equivalent.
% \begin{enumerate}[(i)]
%     % \item $\sup_{z \in \Gamma} \| \conj{\partial_z}^m k^b_z\|_b < \infty$.
%     \item For every $f \in \hb$, the non-tangential limit \[ f^{(m)}(1) := \lim_{z \to 1, z \in \Gamma} f^{(m)}(z)\] exists.
%     \item There exists $\alpha \in \T$ for which
%     \[\int_\T \frac{d\mu_\alpha(\zeta)}{|1-\zeta|^{2m+2}} <\infty.\]
% \end{enumerate}
% \end{thm}
% The last part actually holds for any $\alpha \in \T$ not equal to the non-tangential boundary value $b(1)$ (see \cite[Chapter VI]{sarasonbook} for details). Moreover, Sarason outlines a proof that the weak non-tangential convergence of kernels implicit in point $(i)$ of his theorem can be in fact improved to norm convergence, consistently with the above remark on the Fricain-Mashreghi Theorem.

We wish to mention also the work of Bolotnikov-Kheifets, in which a more direct higher order generalization of the Carathéodory-Julia Theorem has been obtained. In their result \cite[Theorem 1.2]{bolotnikov2006higher}, the existence of non-tangential limit for a higher order derivative of $b$ and positivity of a certain matrix is shown to be equivalent to the boundedness of the supremum of the expressions in \eqref{E:HigherOrderKernelNorm} over a non-tangential approach region. Recently, an interesting generalization of the Carathéodory-Julia Theorem to the setting of arbitrary reproducing kernels and arbitrary domains $X$ instead of $\D$ appeared in the work of Dahlin in \cite{dahlin2024boundary}.

\subsection{More general approach regions. Work of Duan-Li-Mashreghi.}

A different generalization of the Carathéodory-Julia theory arises when replacing the non-tangential approach region $\Gamma$ with a more general one. In our article, an \textit{approach region at $1$} will be an open simply connected domain $\Omega \subset \D$ which satisfies the boundary intersection property $\partial \Omega \cap \T = \{1\}$. Being interested in regions at least as large as the non-tangential region, we shall for convenience also assume that $\Omega$ is star-shaped with respect to $1$ in a neighbourhood of that point, by what we mean that for all $z \in \Omega$ sufficiently close to $1$, the line segment between $1$ and $z$ is contained in $\Omega$. This requirement will exclude some pathological examples which are outside of the scope of our investigation. 

An approach region can be associated to an increasing and continuous function $\rho$ that is defined for non-negative $x$, and satisfies $\rho(0) = 0$, $\rho(x) > 0$ for $x > 0$. Consider a domain of the form
\begin{equation}
    \label{E:OmegaRhoDef}
    \Omega_\rho = \{ z \in \D : 1-|z| > \rho( |\arg z| )\}.
\end{equation} Then $\rho(x) = cx$ corresponds to the non-tangential regions $\Gamma(c)$, while $\rho(x) = cx^2$ corresponds to the usual oricyclic approach regions, which in a neighbourhood of the point $1$ resemble a disk inside $\D$ which touches $\T$ at $1$.

It is natural to ask if analogs of the earlier stated theorems exist for more general approach regions. Specifically, one may try to obtain conditions on $b$ so that the boundary limit
\begin{equation}
    \label{E:BoundaryLimOmega}
    f^{(m)}(1) := \lim_{\substack{z \to 1 \\ z \in \Omega}} f^{(m)}(z)
\end{equation} exists for all $f \in \hb$. The limit must, of course, coincide with the ordinary non-tangential one, if both exist, so our above notation is consistent. This particular question was investigated in the work of Leung in \cite{leung1978boundary}, in which various results are obtained in the case that $b$ is an inner function and the approach regions are of the form \eqref{E:OmegaRhoDef} with $\rho(x) = cx^\gamma$ for some $\gamma \geq 1$. A related study was carried out by Protas in \cite{protas1972tangential}. The first-order results (i.e, $m = 0$) in Leung's work have been recently generalized by Duan-Li-Mashreghi in \cite{duan2025reproducing} from the regions $\Omega_\rho$ considered by Leung to general approach regions. It requires a little bit of preparation to state the main theorem in \cite{duan2025reproducing}. For $z \in \D$ close to $1$ and with positive imaginary part ($\arg z \in (0, \pi)$, say) consider the arc 
\begin{equation}
    \label{E:EzDef} E_z := \{ e^{it} \in \T : \arg z/2 < t < 2 \arg z \}.
\end{equation} For $z$ with negative imaginary part, let $E_z$ be the reflection in the real axis of the arc $E_{\conj{z}}$, and set $E_z$ to be empty for $z = r \in (0,1)$. Define also the circle sector 
\begin{equation}
    \label{E:SEzDef}
    S_z = \{ re^{it} \in \D : e^{it} \in E_z \}.
\end{equation}
The Duan-Li-Mashreghi result reads as follows (see \cite[Theorem 2.1]{duan2025reproducing} for a slightly extended formulation).

\begin{thm}{\textbf{(Duan-Li-Mashreghi)}} \thlabel{T:DLMTheorem} Let $b = BS_\nu$ be an inner function with zero set $(a_n)_n$, and let $\Omega$ be an approach region at $1$. Then the following conditions are equivalent.
\begin{enumerate}[(i)]
    \item \[\sup_{z \in \Omega} \| k^b_z\|_b < \infty.\]
    \item For every $f \in \hb$, $f(z)$ has a limit as $z \to 1$ in $\Omega$.
    \item We have 
    \begin{equation}
        \label{E:AhernClarkCond}
        \sum_n \frac{1-|a_n|^2}{|1-a_n|^{2}} + \int_\T \frac{d\nu(\zeta)}{|1 - \zeta|^{2}} < \infty
    \end{equation} 
    and
    \begin{equation}
        \label{E:FzCond} \sup_{z \in \Omega} \sum_{n : a_n \in S_z} \frac{1-|a_n|^2}{|1-\conj{a_n}z|^{2}} + \int_{E_z} \frac{d\nu(\zeta)}{|1 - \conj{\zeta}z|^{2}} < \infty.
    \end{equation} 
\end{enumerate}
\end{thm}

Condition $(iii)$ above is stated for our convenience in a  slightly different form than the original condition in the statement of \cite[Theorem 2.1]{duan2025reproducing} (see the proof given in the reference for details). Condition \eqref{E:AhernClarkCond} appears already in the context of non-tangential approach regions in \cite{ahern1970radial}, and condition \eqref{E:FzCond} is what has to be added to compensate in the case of an approach region larger than the non-tangential one. The condition $(iii)$ above is easily seen to be equivalent to
\begin{enumerate}
    \item[\textit{(iii')}]  \[ \sup_{z \in \Omega} \sum_{n} \frac{1-|a_n|^2}{|1-\conj{a_n}z|^{2}} + \int_{\T} \frac{d\nu(\zeta)}{|1 - \conj{\zeta}z|^{2}} < \infty.\]
\end{enumerate} 
Indeed, this condition is clearly stronger than \eqref{E:FzCond}, and it implies \eqref{E:AhernClarkCond} by Fatou's lemma. In the other direction, we note that the sets $E_z$ and $S_z$ satisfy the following readily established estimate (see \cite[Lemma 3.4]{duan2025reproducing}) from which the mentioned equivalence readily follows: given two points $w \in \cD := \T \cup \D$ and $z \in \D$, if $w \not\in S_z \cup E_z$, then
\begin{equation}
    \label{E:EzEstimate}
    \frac{1}{|1-\conj{w}z|} \leq \frac{4}{|1-w|}.
\end{equation}
Upside of $(iii)$ over $(iii')$ is that \eqref{E:FzCond} isolates the information needed to conclude the norm convergence of $k^b_z \to k^b_1$ as $z \to 1$ in $\Omega$, which is not automatic in the context of more general approach regions considered here. Indeed, according to \cite[Theorem 2.2]{duan2025reproducing}, under the equivalent conditions of the above theorem, this norm convergence occurs if and only if the quantity in \eqref{E:FzCond} tends to $0$ as $z \to 1$:
\begin{equation}
    \label{E:FzCondLim0} \lim_{\substack{z \to 1 \\ z \in \Omega}} \sum_{n : a_n \in S_z} \frac{1-|a_n|^2}{|1-\conj{z}a_n|^{2}} + \int_{E_z} \frac{d\nu(\zeta)}{|1 - \conj{z}\zeta|^{2}} = 0.
\end{equation}

In the non-tangential case $\Omega = \Gamma$, it can be proved that \eqref{E:AhernClarkCond} implies \eqref{E:FzCondLim0} (and a fortiori \eqref{E:FzCond}). It follows that norm and weak convergence $k^b_z \to k^b_1$ as $z \to 1$ in $\Gamma$ are equivalent.

\section{Main results}

\subsection{Boundary limits for general approach region and symbol.}

The first of our main results is an extension of \thref{T:DLMTheorem} by Duan-Li-Mashreghi to symbols with a non-trivial outer part, and to higher order.

\begin{mainthm} \thlabel{T:maintheorem1}
    Let $b: \D \to \D$ be an analytic function with zero set $(a_n)_n$ and singular inner factor $S_\nu$. Let $\Omega$ be an approach region at $1$. For any integer $m \geq 0$, the following conditions are equivalent.
    \begin{enumerate}[(i)]
        \item \[\sup_{z \in \Omega} \|\conj{\partial_z}^m k^b_z\|_b < \infty.\]
        % \item For some $r \in (0,1)$, we have \[ \sup_{z \in \Omega, |z| > r} \partial_z^m \conj{\partial_z}^m \frac{\log|b(z)|}{1-|z|^2} < \infty.\]
        \item For every $f \in \hb$, $f^{(m)}(z)$ has a limit as $z \to 1$ in $\Omega$.
        \item \[ \sup_{z \in \Omega} \sum_n \frac{1-|a_n|^2}{|1-\conj{a_n}z|^{2m+2}} + \int_\T \frac{d\nu(\zeta)}{|1 - \conj{\zeta}z|^{2m+2}} + \int_\T \frac{-\log|b(\zeta)|}{|1 - \conj{\zeta}z|^{2m+2}} d\m(\zeta) < \infty.\]
    \end{enumerate}
\end{mainthm}

As before, the condition in $(iii)$ can be rewritten as conditions analogous to \eqref{E:AhernClarkCond} and \eqref{E:FzCond}, namely     \[
        \sum_n \frac{1-|a_n|^2}{|1-a_n|^{2m+2}} + \int_\T \frac{d\nu(\zeta)}{|1 - \zeta|^{2m+2}} + \int_\T \frac{-\log |b(\zeta)|}{|1 - \zeta|^{2m+2}} d\m(\zeta) < \infty
\] and
\[ \sup_{z \in \Omega} \sum_{n : a_n \in S_z} \frac{1-|a_n|^2}{|1-\conj{a_n}z|^{2m+2}} + \int_{ E_z} \frac{d\nu(\zeta)}{|1 - \conj{\zeta}z|^{2m+2}} + \int_{E_z} \frac{-\log|b(\zeta)|}{|1 - \conj{\zeta}z|^{2m+2}} d\m(\zeta) < \infty. \]

Our proof of the theorem does not use the technique of Duan-Li-Mashreghi of expressing the kernel norms according to certain formulas of Ahern and Clark from \cite{ahern1970functions} (see Section~3 in \cite{duan2025reproducing}), and can therefore be seen as more elementary. In fact, such a formula for the kernel norm may not exist for the outer function $b_o$. Our approach to overcome this difficulty is described at the beginning of Section~\ref{S:ProofMT1Sec}.

\subsection{Norm convergence of kernels.}

We extend also the above mentioned kernel norm convergence result of Duan-Li-Mashreghi, characterized in the case $m=0$ as equivalent to \eqref{E:FzCondLim0}.

\begin{mainthm} \thlabel{T:maintheorem2}
    Assume that the equivalent conditions in \thref{T:maintheorem1} hold for some $m$. Then the following two conditions are equivalent.
    \begin{enumerate}[(i)]
        \item We have the norm convergence $\conj{\partial_z}^m k^b_z \to \conj{\partial_z}^m k^b_1$ in $\hb$ as $z \to 1$ in $\Omega$.
        \item \[ \lim_{\substack{z \to 1 \\ z \in \Omega}} \sum_{n: a_n \in S_z} \frac{1-|a_n|^2}{|1-\conj{a_n}z|^{2m+2}} + \int_{E_z} \frac{d\nu(\zeta)}{|1 - \conj{\zeta}z|^{2m+2}} + \int_{E_z} \frac{-\log|b(\zeta)|}{|1 - \conj{\zeta}z|^{2m+2}} d\m(\zeta) = 0.\]
    \end{enumerate}
\end{mainthm}

Again, in the non-tangential case $\Omega = \Gamma$, the condition in $(ii)$ of \thref{T:maintheorem2} is automatically satisfied if the equivalent conditions of \thref{T:maintheorem1} are satisfied. Thus we have another proof of the fact that weak and strong convergence $\conj{\partial_z}^m k^b_z \to \conj{\partial_z}^m k^b_1$ as $z \to 1$ in $\Gamma$ are equivalent.

The following reformulation of the result may shed more light on the matter. Introduce the non-negative finite Borel measure $M$ on $\cD$ given by
\[ dM = \sum_n (1-|a_n|^2) d\delta_{a_n} + d\nu - \log |b| d\m,\] where $\delta_{a_n}$ is a unit mass at the point $a_n \in \D$. Then the condition $(iii)$ in \thref{T:maintheorem1} simply means that the functions
\begin{equation}
    \label{E:CauchyKernelPowerEq}
    w \mapsto \frac{1}{|1-\conj{w}z|^{2m+2}}, \quad z \in \Omega
\end{equation} form a bounded subset of the Lebesgue space $L^1(M)$. Although we shall not go into the details of a proof, we wish to mention that one can show that the condition $(ii)$ of \thref{T:maintheorem2} is equivalent to the \textit{uniform integrability} of the family \eqref{E:CauchyKernelPowerEq} with respect to the measure $M$, which by definition means that to any $\epsilon > 0$ there corresponds a $\delta > 0$ such that if $E \subset \cD$ is a Borel set satisfying $M(E) < \delta$, then
\[ \sup_{z \in \Omega} \int_E \frac{1}{|1-\conj{w}z|^{2m+2}} dM(w) < \epsilon.\]
Due to the pointwise convergence \[\frac{1}{|1-\conj{w}z|^{2m+2}} \to \ \frac{1}{|1-\conj{w}|^{2m+2}}\] as $z \to 1$, by basic measure theory this condition is further equivalent to the compactness of the family of functions \eqref{E:CauchyKernelPowerEq} in $L^1(M)$. In this way we connect the properties of families of kernel functions $\conj{\partial_z}^m k^b_z$ in $\hb$ with the properties of families of functions \eqref{E:CauchyKernelPowerEq} in $L^1(M)$.

% If $\Omega = \Gamma(c)$ is a non-tangential region, then our \thref{T:maintheorem1} can be strengthened and stated in a simplified way. 

% \begin{cor*}
%     If $\Omega = \Gamma$ is a non-tangential region, the conditions listed in \thref{T:maintheorem2} are automatically satisfied if the conditions in \thref{T:maintheorem1} are satisfied. In particular, the following two statements are equivalent.
%     \begin{enumerate}[(i)]
%     \item[(i)] As $z \to 1$ within $\Gamma$, we have the norm convergence $\conj{\partial_z}^m k^b_{z} \to \conj{\partial_z}^m k^b_1$ in $\hb$.
%         \item[(ii)] \[ \sum_n \frac{1-|a_n|^2}{|1-a_n|^{2m+2}} + \int_\T \frac{d\nu(\zeta)}{|1 - \zeta|^{2m+2}}  + \int_\T \frac{-\log(|b(\zeta)|)}{|1 - \zeta|^{2m+2}} d\m(\zeta) < \infty.\]
%     \end{enumerate}
% \end{cor*}

% Our corollary is a non-tangential generalization of the earlier stated Fricain-Mashreghi norm convergence result from \cite{fricain2008integral}. It also augments the equivalent conditions in the main result \cite[Theorem 1.2]{bolotnikov2006higher} of Bolotnikov-Khefeits with the integral conditions in the second condition in the corollary.

\subsection{Carathéodory-Julia Theorem for general approach regions: failure and a positive result.}

Given that Sarason's extension of Carathéodory-Julia Theorem has an analog in the context of more general approach regions, one may wonder if the classical theorem also generalizes in a similar way. 
Our main result in this direction shows that a natural restatement of the original theorem of Carathéodory-Julia is not valid in the broader context. To wit, consider the domains $\Omega_\rho$ in \eqref{E:OmegaRhoDef}. Intuitively speaking, a domain of this form will be strictly larger than a non-tangential region in the vicinity of the point $1$ if we have $\lim_{x \to 0} \rho'(x) = 0$. An easy exercise shows that if $\rho$ is convex, then this last condition is actually equivalent to $\Omega_\rho$ not being contained in any non-tangential region $\Gamma(c)$ for any $c$ (recall \eqref{E:GammaRegionDef}). The following theorem shows the failure of a direct generalization of the classical result to general approach regions.

\begin{mainthm}\thlabel{T:maintheorem3}
    Assume that $\rho$ is continuously differentiable and increasing on $(0,\infty)$, $\rho(x) > 0$ for $x > 0$, and 
    \[ \lim_{x \to 0} \rho'(x) = 0.\]
    Then there exists $b: \D \to \D$ such that \[\sup_{z \in \Omega_\rho} \|k^b_z\|_b < \infty\] but \[ \lim_{\substack{z \to 1 \\ z \in \Omega_\rho}} b'(z)\] does not exist.
\end{mainthm}

However, as a consequence of \thref{T:maintheorem2} we can prove the following partial result.

\begin{mainthm}\thlabel{T:maintheorem4}
    Let $\Omega$ be an approach region at $1$. If we have norm convergence $k^b_z \to k^b_1$ in $\hb$ as $z \to 1$ in $\Omega$, then the limits \[b'(1) = \lim_{\substack{z \to 1 \\ z \in \Omega}} b'(z), \quad b(1) = \lim_{\substack{z \to 1 \\ z \in \Omega}} b(z) \] exist, and $|b(1)| = 1$.
\end{mainthm}

We do not know if the converse of the result holds for any approach region $\Omega$ essentially larger than the non-tangential ones, and we have not produced any counter-example. We therefore leave the following question unanswered.

\begin{question*}
For a general approach region $\Omega$, is there a way to characterize the existence of the limits \[b'(1) = \lim_{\substack{z \to 1 \\ z \in \Omega}} b'(z), \quad b(1) = \lim_{\substack{z \to 1 \\ z \in \Omega}} b(z), \quad |b(1)| = 1,\] in terms of some type of behavior of the kernel functions $k^b_z$, $z \in \Omega$?
\end{question*}

% \begin{conjecture*}
%     Let $\Omega$ be an approach region at $1$. The following two statements are equivalent.
%     \begin{enumerate}[(i)]
%         \item As $z \to 1$ within $\Omega$, we have the norm convergence $k^b_{z} \to k^b_1$ in $\hb$.
%         \item The limits \[b'(1) = \lim_{\substack{z \to 1 \\ z \in \Omega}} b'(z), \quad b(1) = \lim_{\substack{z \to 1 \\ z \in \Omega}} b(z) \] exist, and $|b(1)| = 1$.
%     \end{enumerate}
% \end{conjecture*}

\section{Proof of Theorem A}
\label{S:ProofMT1Sec}

It will be convenient to divide the proof into the two special cases of $b$ being a Blaschke product and $b$ being zero-free in $\D$. The result in general will be shown to follow from these two special cases.

The proof of the equivalence of conditions $(i)$ and $(iii)$ in \thref{T:maintheorem1} is the most involved part, but it will be established in a natural way. In the case that $b$ is a Blaschke product, certain observations and the Leibniz rule will express $\partial^m_z \conj{\partial_z}^m \|k^{b}_z\|_{b}^2 = \|\conj{\partial_z}^m k^{b}_z\|_{b}^2$ as a sum of various terms involving derivatives of $b$ or its factors, and a single positive quantity of magnitude proportional to the terms appearing in $(iii)$ of \thref{T:maintheorem1}. Estimation of the derivatives of $b$ and its factors in \thref{L:bjSupBoundLemma} and \thref{L:bZeroFreeSupBoundLemma} will show that the positive quantity is the principal one. The complementary case of zero-free $b$ will be treated in a way not much different.

\subsection{Blaschke product case.} \label{S:Theorem1ProofBlaschkeSec}

For the remainder of this section, let $b$ be the Blaschke product corresponding to the zero sequence $(a_n)_{n \geq 1}$. The condition $(iii)$ in \thref{T:maintheorem1} is then 
\begin{equation}
    \label{E:PartIIIforBlaschke}
    \sup_{z \in \Omega} \sum_{j=1}^\infty \frac{1-|a_j|^2}{|1-\conj{a_j}z|^{2m+2}} < \infty.
\end{equation} Let $b_j(z)$ be the finite Blaschke product product corresponding to $(a_n)_{n=1}^j$ ($b_0 \equiv 1)$. 
For a product $b = B_1B_2$, we have the readily-established decomposition formula
\begin{equation}
    \label{E:KernelDecompFormula}
    \|k^b_z\|_b^2 =  \|k^{B_1}_z\|_{B_1}^2 + |B_1(z)|^2\|k^{B_2}_z\|_{B_2}^2.
\end{equation}
Recall also the well-known identity
\begin{equation}
        \label{E:MagicFormula} 1 - \frac{|z-a|^2}{|1-\conj{a}z|^2} = \frac{(1-|z|^2)(1-|a|^2)}{|1-\conj{a}z|^2}, \quad a,z \in \D.
    \end{equation}
In the case of a Blaschke factor \[ B_1(z) = \frac{|a|}{a}\frac{a-z}{1-\conj{a}z}, \quad a,z \in \D,\] we note that \eqref{E:MagicFormula} implies
\[ \|k^{B_1}_z\|_{B_1}^2 = \frac{1-|a|^2}{|1-\conj{a}z|^2}.\]
Now, an iteration of the above decomposition formula leads to
\begin{equation}
    \label{E:BlaschkeNormFormula}
    \|k^b_z\|_b^2 = \sum_{j \geq 1} |b_{j-1}(z)|^2 \frac{1-|a_j|^2}{|1-\conj{a_j}z|^2}.
\end{equation} As a consequence of the Blaschke condition on $(a_n)_n$, it is not hard to see that the derivatives of $\|k^b_z\|^2_b$ can be computed by term-wise differentiation of the series in \eqref{E:BlaschkeNormFormula}. An application of the Leibniz rule and an exchange of the order of summations shows that 

\begin{equation}
    \label{E:BlaschkeDerivativeKernelNormFormula}
    \partial^m_z \conj{\partial_z}^m \|k^{b}_z\|_{b}^2 = \sum_{k,l = 0}^m C_{k,l}(z)
\end{equation} where
\begin{equation}
    \label{E:CklDef}
    C_{k,l}(z) = \sum_{j \geq 1} {m \choose k} {m \choose l} b^{(m-k)}_{j-1}(z) \conj{b^{(m-l)}_{j-1}(z)} \conj{a_j}^k a_j^l k!l!\frac{1-|a_j|^2}{(1-\conj{a_j}z)^{1+k}(1-a_j\conj{z})^{1+l}}
\end{equation} 

According to the next lemma, the coefficients of the terms $\frac{1-|a_j|^2}{(1-\conj{a_j}z)^{1+k}(1-a_j\conj{z})^{1+l}}$ in the expression for $C_{k,l}(z)$ are well-behaved.

\begin{lem} \thlabel{L:bjSupBoundLemma} Let $b$ and $b_j$ be Blaschke products as above.

\begin{itemize}
    \item If $(i)$ in \thref{T:maintheorem1} holds, then the derivatives up to order $m$ of $b_j$ are uniformly bounded in $\Omega$, and $\lim_{z \to 1, z \in \Omega} |b(z)| = 1$.
    \item If $(iii)$ in \thref{T:maintheorem1} holds, then the derivatives up to order $2m+1$ of $b_j$ are uniformly bounded in $\Omega$, and $\lim_{z \to 1, z \in \Omega} |b(z)| = 1$. Moreover, the limits \[\lim_{\substack{z \to 1 \\ z \in \Omega}} b^{(k)}(z), \lim_{\substack{z \to 1 \\ z \in \Omega}} b_j^{(k)}(z), \quad j \geq 0, k = 0, 1, \ldots, 2m\] exist.
\end{itemize}
\end{lem}

\begin{remark}
    Note the asymmetry in the above statement. Since the two conditions $(i)$ and $(iii)$ in \thref{T:maintheorem1} will ultimately be proven to be equivalent, the conclusion in the first part of the lemma can be upgraded to the stronger conclusion of the second part. For the time being, however, the above statement both suffices and is more convenient.
\end{remark}

We will prove the lemma a bit later. Assuming it for now, we see from Cauchy-Schwarz inequality applied to \eqref{E:CklDef} that if either $(i)$ or $(iii)$ in \thref{T:maintheorem1} hold for $b$ and $m$, then there exists a constant $C > 0$  such that
\begin{equation}
    \label{E:CNklEstimate}
|C_{k,l}(z)| \leq C \sqrt{\sum_{j \geq 1} \frac{1-|a_j|^2}{|1-\conj{a_j}z|^{2k+2}}}\sqrt{\sum_{j \geq 1} \frac{1-|a_j|^2}{|1-\conj{a_j}z|^{2l+2}}}, \quad z \in \Omega.
\end{equation} Hence \eqref{E:PartIIIforBlaschke} implies that the left-hand side in \eqref{E:BlaschkeDerivativeKernelNormFormula} is uniformly bounded for $z \in \Omega$. Thus, for Blaschke products $b$, the condition $(iii)$ implies the condition $(i)$ in \thref{T:maintheorem1}. 

Let us now prove the other implication $(i) \Rightarrow (iii)$ in \thref{T:maintheorem1} by induction on $m$. For the base case $m=0$, it suffices to note that by \thref{L:bjSupBoundLemma}, for $z \in \Omega$ sufficiently close to $1$, we have \[ 1/2 \leq |b(z)| \leq |b_{j-1}(z)| \leq 1.\] Hence the case $m=0$ follows from \eqref{E:BlaschkeNormFormula}. For the induction step, we assume that condition $(i)$ implies $(iii)$ for $m-1$, and we suppose $(i)$ holds for $m$. Then the derivative of order $m$ of every function $f \in \hb$ is bounded in $\Omega$, and so the same is true for derivatives of lower order. From principle of uniform boundedness it follows that $\sup_{z \in \Omega} \|\conj{\partial_z}^i k^b_z\|_b < \infty$ for $i \leq m$. By the induction hypothesis, we deduce that
\begin{equation}
    \label{E:IndAssumpt}
    \sup_{z \in \Omega} \sum_{n} \frac{1-|a_n|^2}{|1-\conj{a_n}z|^{2m}} < \infty.
\end{equation}  

The idea is to show that the term $C_{m,m}(z)$ in \eqref{E:BlaschkeDerivativeKernelNormFormula} dominates the other terms. Consider first the terms on the right-hand side of \eqref{E:BlaschkeDerivativeKernelNormFormula} which are indexed by $k < m$, $l < m$. It follows from \eqref{E:IndAssumpt} and the estimate \eqref{E:CNklEstimate}, which is again valid by \thref{L:bjSupBoundLemma}, that
\begin{equation}
    \label{E:CNklLessThanmEstimate}
    \sup_{z \in \Omega} \sum_{k,l = 0}^{m-1}| C_{k,l}(z)| < \infty. 
\end{equation} Now consider the terms on the right-hand side of \eqref{E:BlaschkeDerivativeKernelNormFormula} for which one of the indices $k$ or $l$ is equal to $m$. Similarly, we deduce that
\begin{equation}
    \label{E:CNmkEstimate}
    \sum_{k=0}^{m-1} |C_{k, m}(z)| +  \sum_{l=0}^{m-1} |C_{m, l}(z)| \leq D\sqrt{\sum_{j \geq 1} \frac{1-|a_j|^2}{|1-\conj{a_j}z|^{2m+2}}}, \quad z \in \Omega,
\end{equation} where $D$ is a constant. The remaining term indexed by $k,l = m$ is
\[ C_{m,m}(z) = \sum_{j \geq 1} |b_j(z)|^2 |a_j|^{2m} (m!)^2\frac{1-|a_j|^2}{|1-\conj{a_j}z|^{2m+2}}.\] We may without loss of generality assume that $|a_j| > 1/2$ holds for all $j$ (removing finitely many points from the sequence $(a_j)_j$ clearly does not affect the validity of any of the equivalent conditions in \thref{T:maintheorem1}). As in our earlier proof of the base case $m=0$, it follows from \thref{L:bjSupBoundLemma} that we have
\begin{equation}
    \label{E:CNmmLowerEst}
    C_{m,m}(z) \geq 2^{-2m-2}(m!)^2 \sum_{j \geq 1} \frac{1-|a_j|^2}{|1-\conj{a_j}z|^{2m+2}}
\end{equation} for all $z \in \Omega$ which are sufficiently close to $1$. Now we may combine \eqref{E:BlaschkeDerivativeKernelNormFormula} with the three estimates \eqref{E:CNklLessThanmEstimate}, \eqref{E:CNmkEstimate} and \eqref{E:CNmmLowerEst} to conclude that for $z \in \Omega$ sufficiently close to $1$, we have 
\[ 2^{-2m-2}(m!)^2  \sum_{j \geq 1} \frac{1-|a_j|^2}{|1-\conj{a_j}z|^{2m+2}} \leq C' + D'\sqrt{\sum_{j \geq 1} \frac{1-|a_j|^2}{|1-\conj{a_j}z|^{2m+2}}} + \partial^m_z \conj{\partial_z}^m \|k^b_z\|_b^2\] where $C'$ and $D'$ are positive constants. If $(i)$ of \thref{T:maintheorem1} holds, then this inequality can hold only if $(iii)$ of \thref{T:maintheorem1} holds. This completes the proof of the equivalence $(i) \Leftrightarrow (iii)$ in the case of Blaschke products (assuming the earlier-stated lemma).

It remains to show that condition $(ii)$ in \thref{T:maintheorem1} is equivalent to the remaining two. It certainly implies $(i)$, by the principle of uniform boundedness and \eqref{E:DerivativReproducingFormula}. Conversely, if $(i)$ and $(iii)$ hold, then from the second part of \thref{L:bjSupBoundLemma} and and from \eqref{E:HbKernelEq} we conclude that the derivative of order $m$ of any finite combination of kernels functions in $\hb$ have a limit as $z \to 1$ in $\Omega$. Thus $\lim_{z \to 1, z \in \Omega} f^{(m)}(z) = \lim_{z \to 1, z \in \Omega} \ip{f}{\conj{\partial_z}^m k^b_z}_b$ exists for a dense set of functions $f$ in $\hb$. This together with $\sup_{z \in \Omega} \|\conj{\partial_z}^m k^b_z\|_b$ means that $\conj{\partial_z}^m k^b_z$ converges weakly as $z \to 1$ in $\Omega$, which is clearly equivalent to the condition $(ii)$ in \thref{T:maintheorem1}.

\subsection{Zero-free case.}

Let us now consider the case of a zero-free symbol $b$. In \eqref{E:bInnerOuterFact}, set $d\mu := d\nu - \log |b| d\m$, so that
\begin{equation}
    \label{E:zeroFreeBFormula}
    b(z) = \exp \Bigg( -\int_\T \frac{\zeta + z}{\zeta - z} d\mu(\zeta)\Bigg), \quad z \in \D.
\end{equation}
Condition $(iii)$ in \thref{T:maintheorem1} becomes
\begin{equation}
    \label{E:PartIIIforZeroFree}
    \sup_{z \in \Omega} \int_\T \frac{d\mu(\zeta)}{|1-\conj{\zeta}z|^{2m+2}} < \infty. 
\end{equation}
Since $\Re \frac{\zeta + z}{\zeta - z} = \frac{1-|z|^2}{|1-\conj{\zeta}z|^2}$, \eqref{E:zeroFreeBFormula} implies that
\begin{equation}
    \label{E:LogExprZeroFree}
    \frac{-\log |b(z)|^2}{2(1-|z|^2)} = \int_\T \frac{d\mu(\zeta)}{|1-\conj{\zeta}z|^{2}}.
\end{equation} Since $-\log x = \sum_{n=1}^\infty \frac{(1-x)^n}{n}$ for $x \in (0,1)$, we see from this expression and from \eqref{E:HbNormEq} that
\begin{equation}
    \label{E:muSKbExpression}
    \int_\T \frac{d\mu(\zeta)}{|1-\conj{\zeta}z|^{2}} = S(|b(z)|^2)\|k^b_z\|_b^2
\end{equation} where
\begin{equation}
    \label{E:Sformula} S(x) := \sum_{n=1}^\infty \frac{(1-x)^{n-1}}{2n}, \quad x \in (0, 1].
\end{equation} Clearly $S$ is a smooth function on $(0,1)$, and $\lim_{x \to 1} S^{(k)}(x)$ exists for each non-negative integer $k$. Moreover, $S$ is bounded from below by $1/2$. Applying the operator $\partial^m_z \conj{\partial_z}^m$ and the Leibniz rule to \eqref{E:muSKbExpression}, we see that
\begin{equation}
    \label{E:muSKbExpressionLEibniz}
    \int_\T \frac{(m!)^2d\mu(\zeta)}{|1-\conj{\zeta}z|^{2m+2}} = \sum_{i,j=0}^m P_{i,j}(z) \cdot \partial^i_z \conj{\partial_z}^j \|k^b_z\|_b^2
\end{equation} where $P_{i,j}(z)$ is a product of derivatives of $S$ evaluated at $|b(z)|^2$, and of derivatives of $b$, all up to order $m$. Moreover, $P_{0,0}(z) = S(|b(z)|^2)$. 

We will need a variant of \thref{L:bjSupBoundLemma}.

\begin{lem} \thlabel{L:bZeroFreeSupBoundLemma} Let $b$ be zero-free in $\D$.

\begin{itemize}
    \item If $(i)$ in \thref{T:maintheorem1} holds, then the derivatives up to order $m$ of $b$ are uniformly bounded in $\Omega$, and $\lim_{z \to 1, z \in \Omega} |b(z)| = 1$.
    \item If $(iii)$ in \thref{T:maintheorem1} holds, then the derivatives up to order $2m+1$ of $b$ are uniformly bounded in $\Omega$, and $\lim_{z \to 1, z \in \Omega} |b(z)| = 1$. Moreover, the limits \[\lim_{\substack{z \to 1\\ z \in \Omega}} b^{(k)}(z), \quad k = 0, 1, \ldots, 2m\] exist.
\end{itemize}
\end{lem}

As in the case of the corresponding lemma for Blaschke products, we leave the proof for later. Note that
\begin{equation}
    \label{E:laplacianKernelCauchySchwarz}
    |\partial^i_z \conj{\partial_z}^j \|k^b_z\|_b^2| = |\ip{\conj{\partial_z}^i k^b_z}{\conj{\partial_z}^j k^b_z}_b| \leq \|\conj{\partial_z}^i k^b_z\|_b \cdot \|\conj{\partial_z}^j k^b_z\|_b.
\end{equation} As before, the condition $(i)$ in \thref{T:maintheorem1} implies that $\sup_{z \in \Omega} \|\conj{\partial_z}^i k^b_z\|_b < \infty$ for $i \leq m$. Then \eqref{E:laplacianKernelCauchySchwarz} together with \thref{L:bZeroFreeSupBoundLemma} show that the left-hand side in \eqref{E:muSKbExpressionLEibniz} is bounded for $z \in \Omega$, and we have proved the implication $(i) \Rightarrow (iii)$ in \thref{T:maintheorem1} for zero-free $b$. 
Conversely, let us prove the implication $(iii) \Rightarrow (i)$ by induction on $m$. The base case $m=0$ is obvious from \eqref{E:muSKbExpression}, the second statement in \thref{L:bZeroFreeSupBoundLemma} and the boundedness from above and below of $S$ for $x \in (1/2, 1)$. Assuming that the implication $(iii) \Rightarrow (i)$ in \thref{T:maintheorem1} is valid for $m-1$, we see in \eqref{E:muSKbExpressionLEibniz} that
\[ \int_\T \frac{(m!)^2d\mu(\zeta)}{|1-\conj{\zeta}z|^{2m+2}} =  S(|b(z)|^2) \cdot \|\conj{\partial_z}^m k^b_z\|_b^2 + D(z) \cdot \|\conj{\partial_z}^m k^b_z\|_b + \mathcal{O}(1)\] where $D(z)$ is bounded for $z \in \Omega$. As the left-hand side is assumed bounded in $\Omega$, so must be the right-hand side. Since $S(|b(z)|^2)$ is bounded from below in $\Omega$, we deduce that $\sup_{z \in \Omega} \|\conj{\partial_z}^m k^b_z\|_b < \infty$.

The proof that the conditions $(i)$ and $(iii)$ in \thref{T:maintheorem1} are equivalent to $(ii)$ in the case of zero-free $b$ is analogous to the corresponding proof for the case of the Blaschke product. The only difference is that we use the second part of \thref{L:bZeroFreeSupBoundLemma} instead of the second part of \thref{L:bjSupBoundLemma}. We leave out the details.

\subsection{Combining the special cases into the general result.}
\label{S:T2ProofCombiningSubsec}

In the last two sections we have proved \thref{T:maintheorem1} in the two cases that $b = B_1$ is a Blaschke product or $b = B_2$ is zero-free in $\D$. Let us now show that these two cases imply the result for the general symbol $b = B_1B_2$.

Assume that $(i)$ in \thref{T:maintheorem1} holds for $b = B_1B_2$, where $B_1$ is a Blaschke product and $B_2$ is zero-free in $\D$. It is well known that we have the contractive containments $\h(B_1) \subset \hb$, $\h(B_2) \subset \hb$, and that $f \in \hb$ can be decomposed as
\begin{equation}
    \label{E:HbFuncDecomp}
    f = B_2f_1 + f_2, \quad f_1 \in \h(B_1), \, f_2 \in \h(B_2)
\end{equation} (see, for instance, \cite[page 5]{sarasonbook} or \cite[Theorem 16.23]{hbspaces2fricainmashreghi}). 
The contractive containment $\h(B_1) \subset \hb$ implies the set containment
\[ \{ f \in \h(B_1) : \|f\|_{B_1} = 1\} \subset \{ f \in \hb : \|f\|_b \leq 1 \} \] and therefore
\begin{align} \label{E:FunctionalNormSubspaceEstimate}
    \|\conj{\partial_z}^m k^{B_i}_z\|_{B_i} &= \sup_{\substack{f \in \h(B_1)\\ \|f\|_{B_1} = 1}} |f^{(m)}(z)| \\
    &\leq \sup_{\substack{f \in \h(b) \\ \|f\|_{b} \leq 1}} |f^{(m)}(z)| \nonumber \\
    &= \| \conj{\partial_z}^m k^b_z\|_b. \nonumber
\end{align} 
Hence by the special cases of \thref{T:maintheorem1}, $f_i^{(k)}(z)$ has a limit as $z \to 1$ in $\Omega$ for $k \leq m$, $i=1,2$. The same is true for the zero-free function $B_2$, according to \thref{L:bZeroFreeSupBoundLemma}. By \eqref{E:HbFuncDecomp}, we see that $f^{(m)}(z)$ has a limit as $z \to 1$ in $\Omega$. The proof in the general case of the implication $(i) \Rightarrow (ii)$ in \thref{T:maintheorem1} is thus complete.

We prove $(ii) \Rightarrow (iii)$. The containments $\h(B_1) \subset \hb$, $\h(B_2) \subset \hb$ show that $(ii)$ in \thref{T:maintheorem1} holds for the symbols $B_1$ and $B_2$, and therefore so does $(iii)$, which means that \eqref{E:PartIIIforBlaschke} and \eqref{E:PartIIIforZeroFree} are satisfied. Summing the conditions, we obtain $(iii)$ in \thref{T:maintheorem1} for the symbol $b = B_1B_2$. 

Finally, let us prove the remaining implication $(iii) \Rightarrow (i)$ in the general case. The special cases imply that
\[ \sup_{z \in \Omega} \| \conj{\partial_z}^l k^{B_i}_z\|_{B_i} < \infty, i =1,2, \quad l = 0, 1, \ldots, m.\] It suffices to apply the Leibniz formula, the inequality \eqref{E:laplacianKernelCauchySchwarz} and either \thref{L:bjSupBoundLemma} or \thref{L:bZeroFreeSupBoundLemma} to the decomposition in \eqref{E:KernelDecompFormula} to conclude that $\sup_{z \in \Omega} \| \conj{\partial_z}^m k^{b}_z\|_{b} < \infty$. The proof is complete.

\subsection{Proofs of the auxiliary lemmas.}

\begin{proof}[Proof of \thref{L:bjSupBoundLemma}]
We start by assuming the condition $(i)$ in \thref{T:maintheorem1}. We use induction on $m$ to prove that the $b_j^{(m)}$ are uniformly bounded in $\Omega$, with the base case $m=0$ being trivial. Let $m > 0$ and assume that $\sup_{z \in \Omega} \| \conj{\partial_z}^m k^b_z\|_b < \infty$. As before, we conclude that $\sup_{z \in \Omega} \| \conj{\partial_z}^i k^b_z\|_b < \infty$ for $i \leq m$, and by the induction hypothesis derivatives of $b_j$ up to order $m-1$ are uniformly bounded in $\Omega$. Fix $\lambda \in \D$ for which $b(\lambda) \neq 0$. The term $\frac{1}{1-\conj{\lambda}z}$ and all of its derivatives are bounded in $\D$, hence in $\Omega$. Since $b$ and $b_j$ are inner functions, the space $\h(b_j)$ lies isometrically embedded in $\hb$. Therefore by \eqref{E:HbNormEq}, we have
\[ \|k^{b_j}_\lambda\|_b^2 = \|k^{b_j}_\lambda\|_{b_j}^2 = \frac{1-|b_j(\lambda)|^2}{1-|\lambda|^2}\leq \frac{1-|b(\lambda)|^2}{1-|\lambda|^2} \leq \|k^b_\lambda\|_b^2\] and so
\[ |\partial^m_z k^{b_j}_\lambda(z)| = |\ip{k^{b_j}_\lambda}{\conj{\partial_z}^m k^b_z}_b| \leq \| k^{b_j}_\lambda\|_b \|\conj{\partial_z}^m k^b_z\|_b \leq \| k^{b}_\lambda\|_b \|\conj{\partial_z}^m k^b_z\|_b.\] The rightmost quantity is uniformly bounded for $z \in \Omega$. Leibniz rule applied to the expression for $k^b_\lambda$ in \eqref{E:HbKernelEq} shows that
\begin{equation}
    \label{E:partialMkernelEqRest}
    \partial^m_z k^{b_j}_\lambda(z) = \frac{-\conj{b_j(\lambda)}b_j^{(m)}(z)}{1-\conj{\lambda}z} + r_j(z)
\end{equation} where $r_j(z)$ is a linear combination of products of derivatives of $b_j$ of order strictly less than $m$ and powers of $(1-\conj{\lambda}z)^{-1}$. By induction we have that $r_j$ is uniformly bounded in $\Omega$. Since $|b_j(\lambda)| \geq |b(\lambda)| > 0$, it follows from \eqref{E:partialMkernelEqRest} that $b_j^{(m)}$ is also uniformly bounded in $\Omega$. The fact that condition $(i)$ in \thref{T:maintheorem1} implies $\lim_{z \to 1, z \in \Omega} |b(z)| = 1$ follows at once from $\sup_{z \in \Omega} \|k^b_z\|_b < \infty$ and \eqref{E:HbNormEq}. We have proved the first point in \thref{L:bjSupBoundLemma}.

Let us now prove the second point, so instead assume $(iii)$ in \thref{T:maintheorem1}. Using \eqref{E:MagicFormula} and our hypothesis, we note that for $z \in \Omega$ we have
\begin{equation}
    \label{E:MagicFormulaEstimate}
    \sup_n \, 1 - \frac{|z-a_n|^2}{|1-\conj{a_n}z|^2} = \sup_n \frac{(1-|z|^2)(1-|a_n|^2)}{|1-\conj{a_n}z|^2} \leq \sum_j \frac{(1-|z|^2)(1-|a_j|^2)}{|1-\conj{a_j}z|^2} = \mathcal{O}(1-|z|^2). 
\end{equation} We conclude that there exists $r \in (0,1)$ such that for $z \in \Omega$ with $|z| > r$, we have the lower bound
\begin{equation}
    \label{E:ZerosFarAwayBound}
    1/2 \leq \frac{|z-a_n|}{|1-\conj{a_n}z|}.
\end{equation} Differentiating the expression
\[ \log |b_j(z)| = \frac{1}{2}\sum_{n=1}^j \log \bigg( \frac{|z-a_n|^2}{|1-\conj{a_n}z|^2}\bigg)\] results in
\begin{equation}
    \label{E:BlaschkeLogarithmicDerivative}
    \frac{b_j'(z)}{b_j(z)} = \partial_z \log |b_j(z)|^2 = \sum_{n=1}^j \frac{1-|a_n|^2}{(z-a_n)(1-\conj{a_n}z)}.
\end{equation}  By \eqref{E:ZerosFarAwayBound} we obtain that
\[ |b_j'(z)| \leq 2\sum_{n \geq 1} \frac{1-|a_n|^2}{|1-\conj{a_n}z|^2}, z \in \Omega, \quad |z| > r.\] For $z$ satisfying $|z| \leq r$, we may use the usual Bloch-type estimate \[ |f'(z)| \leq \frac{2 \sup_{z \in \D} |f(z)|}{1-|z|}\] to conclude the uniform boundedness of $b_j'$ in $\Omega$. The case $m=0$ of our proof is complete.
To prove the cases $m \geq 1$, simply multiply \eqref{E:BlaschkeLogarithmicDerivative} through by $b_j(z)$ and differentiate $k$ times to see that $b_j^{(k+1)}(z)$ is a finite sum of products of derivatives of $b_j$ of order lower than $k$, and of terms controlled by $\sum_{n \geq 1} \frac{1-|a_n|^2}{|1-\conj{a_n}z|^{k+2}}$ in $\Omega$. Then the general case follows readily by induction. Since $b_j \to b$ uniformly on compacts as $j \to \infty$, these estimates show also that $b^{(k)}$ is bounded in $\Omega$ for $k \leq 2m+1$. This means that $b_j^{(k)}$ and $b^{(k)}$ are Lipschitz in $\Omega$ for $k \leq 2m$, and so they extend to continuous functions on $\partial \Omega$. In particular $\lim_{z \to 1, z\in \Omega} b^{(k)}(z)$ and $\lim_{z \to 1, z\in \Omega} b_j^{(k)}(z)$ exist for $k \leq 2m$. The fact that $\lim_{z \to 1, z \in \Omega} |b(z)| = 1$ is yet again a byproduct of the proof. Indeed,
\[ \log |b(z)| = \frac{1}{2}\sum_{n \geq 1} \log \bigg( \frac{|z-a_n|^2}{|1-\conj{a_n}z|^2}\bigg),\] and from the estimate $(x-1)/2 \leq \log x \leq x-1$, $x \in (1/2, 1)$ and \eqref{E:MagicFormula}, we deduce that
\[ \log |b(z)| = \mathcal{O}(1-|z|^2), \quad z \in \Omega.\] Hence $|b(z)| \to 1$ as $z \to 1$ in $\Omega$. 
\end{proof}

\begin{proof}[Proof of \thref{L:bZeroFreeSupBoundLemma}] The proof of the first point is completely analogous to the corresponding part of the proof of \thref{L:bjSupBoundLemma}, with $b=b_j$. To prove the second point, we use \eqref{E:zeroFreeBFormula} to see that
\begin{equation}
    \label{E:LogarithmicDerivativeZeroFree}
    \frac{b'(z)}{b(z)} = -\int_{\T} \frac{2 \zeta d\mu(\zeta)}{(\zeta-z)^2}.
\end{equation} Multiplying through by $b(z)$ and differentiating $2m$ times shows that \eqref{E:PartIIIforZeroFree} (i.e., $(iii)$ of \thref{T:maintheorem1}) implies boundedness of $b^{(2m+1)}$ in $\Omega$. As in the proof of \thref{L:bjSupBoundLemma}, it follows that $\lim_{z \to 1, z\in \Omega} b^{(k)}(z)$ exists for $k \leq 2m$. Finally, we have
\[ -\log |b(z)| = \int_\T \frac{1-|z|^2}{|1-\conj{\zeta}z|^2} d\mu(\zeta)\] and so \eqref{E:PartIIIforZeroFree} shows that $|b(z)| \to 1$ as $z \to 1$ in $\Omega$.
\end{proof}

We state the following result which follows from Lemmas \ref{L:bjSupBoundLemma} and \ref{L:bZeroFreeSupBoundLemma}.
\begin{cor}
    Assume that the equivalent conditions in \thref{T:maintheorem1} hold for some $m$. Then the derivatives up to order $2m+1$ of $b$ are uniformly bounded in $\Omega$, and $\lim_{z \to 1, z \in \Omega} |b(z)| = 1$. Moreover, the limits \[\lim_{\substack{z \to 1\\ z \in \Omega}} b^{(k)}(z), \quad k = 0, 1, \ldots, 2m\] exist.
\end{cor}
We will see in Theorem \ref{T:maintheorem3} that in a general region $\Omega$, $\lim_{z \to 1, z \in \Omega} b^{(2m+1)}(z)$ may not exist.

\section{Proof of Theorem B}
\label{S:ProofMT2Sec}

In this section we prove \thref{T:maintheorem2}, and so throughout we assume that the equivalent conditions of \thref{T:maintheorem1} are satisfied for some fixed integer $m$. Consequently, we often will use in our arguments the various partial results derived in the previous section (such as \thref{L:bjSupBoundLemma} and \thref{L:bZeroFreeSupBoundLemma}).

\subsection{Some initial observations.}

The following is our main simplifying observation. 

\begin{lem} \thlabel{L:TheoremBLemma1}
    Assume that $b$ satisfies the equivalent conditions stated in \thref{T:maintheorem1} for some $m \geq 0$. To conclude the norm convergence $\conj{\partial_z}^m k^b_{z} \to \conj{\partial_z}^m k^b_1$ as $z \to 1$ in $\Omega$, it suffices to show that 
    \begin{equation}
        \label{E:OmegaNormLimitKernel}
        \lim_{\substack{z \to 1 \\ z \in \Omega}} \|\conj{\partial_z}^m k^b_z\|_b
    \end{equation} exists.
\end{lem}

For convenience we will use in the following proof the Fricain-Mashreghi result presented in the introduction (\thref{T:FricainMashreghiTheorem}), although this is not, strictly speaking, necessary.

\begin{proof}[Proof of \thref{L:TheoremBLemma1}]
Suppose that the limit in \eqref{E:OmegaNormLimitKernel} exists. Note that part $(iii)$ in \thref{T:FricainMashreghiTheorem} is satisfied as a consequence of Fatou's lemma and our assumption of $(iii)$ of \thref{T:maintheorem1} being satisfied. Hence all three conditions in \thref{T:FricainMashreghiTheorem} hold. By part $(i)$ of \thref{T:FricainMashreghiTheorem}, we conclude that \[ \lim_{r \to 1-} \|\conj{\partial_z}^m k^b_r\|_b = \|\conj{\partial_z}^m k^b_1\|_b.\] Hence the limit in \eqref{E:OmegaNormLimitKernel} must also equal $\|\conj{\partial_z}^m k^b_1\|_b$. Noting that
\[ \lim_{\substack{z \to 1 \\ z \in \Omega}}  \partial_z^m \conj{\partial_z}^m k^b_1(z) = \|\conj{\partial_z}^m k^b_1\|^2_b\] we conclude the proof as follows: 
\[  \lim_{\substack{z \to 1 \\ z \in \Omega}} \|\conj{\partial_z}^m k^b_z - \conj{\partial_z}^m k^b_1\|_b =         \lim_{\substack{z \to 1 \\ z \in \Omega}} \|\conj{\partial_z}^m k^b_z\|_b^2 + \|\conj{\partial_z}^m k^b_1\|_b - 2\Re \partial_z^m \conj{\partial_z}^m k^b_1(z) = 0.\] 
\end{proof}

The following two lemmas will also be useful.

\begin{lem} \thlabel{L:TheoremBLemma2}
    If we have norm convergence $\conj{\partial_z}^m k^b_{z} \to \conj{\partial_z}^m k^b_1$ as $z \to 1$ in $\Omega$, then also we have norm convergence
    $\conj{\partial_z}^j k^b_{z} \to \conj{\partial_z}^j k^b_1$ as $z \to 1$ in $\Omega$ for $j < m$.
\end{lem}

\begin{proof}
    Recall that we assume that if $z \in \Omega$ close enough to $1$, then the line segment between $z$ and $1$ is contained in $\Omega$. Then, since $|f^{(j)}(z) - f^{(j)}(1)| = |\int_1^z f^{(j+1)}(w)\, dw|$, we get
    \begin{align*}
    \limsup_{\substack{z \to 1 \\ z \in \Omega}} \|\conj{\partial_z}^j k^b_z - \conj{\partial_z}^j k^b_1\|_b &=\limsup_{\substack{z \to 1 \\ z \in \Omega}} \sup_{ f : \|f\|_b = 1} |f^{(j)}(z) - f^{(j)}(1)| \\ &\leq \limsup_{\substack{z \to 1 \\ z \in \Omega}} \| \conj{\partial_z}^{j+1} k^b_z\|_b \cdot |z-1|.
\end{align*}
Thus norm convergence $\conj{\partial_z}^j k^b_{z} \to \conj{\partial_z}^j k^b_1$ as $z \to 1$ in $\Omega$ is implied by mere boundedness of the norms $\| \conj{\partial_z}^{j+1} k^b_z\|_b$ for $z \in \Omega$, and a fortiori by the assumed norm convergence. The claim then follows by iteration down from case $j +1 = m$.
\end{proof}

Recall the definition of the intervals $E_z$ in \eqref{E:EzDef} and the sectors $S_z$ in \eqref{E:SEzDef}.

\begin{lem} \thlabel{L:TheoremBLemma3}
    Assume that the equivalent conditions of \thref{T:maintheorem1} are satisfied for some integer $m$. Then
    \begin{equation}
        \label{E:LimIntCond} \lim_{\substack{z \to 1, z \in \Omega}} \sum_n \frac{1-|a_n|^2}{|1-\conj{a_n}z|^{2m+2}} + \int_\T \frac{d\nu(\zeta)}{|1 - \conj{\zeta}z|^{2m+2}} + \int_\T \frac{-\log|b(\zeta)|}{|1 - \conj{\zeta}z|^{2m+2}} d\m(\zeta)
    \end{equation} exists if and only if  
    \begin{equation}
        \label{E:LimIntCondEz}
        \lim_{z \to 1, z \in \Omega} \sum_{n: a_n \in S_z} \frac{1-|a_n|^2}{|1-\conj{a_n}z|^{2m+2}} + \int_{E_z} \frac{d\nu(\zeta)}{|1 - \conj{\zeta}z|^{2m+2}} + \int_{E_z} \frac{-\log|b(\zeta)|}{|1 - \conj{\zeta}z|^{2m+2}} d\m(\zeta) = 0.
    \end{equation} 
\end{lem}

\begin{proof}
    By condition $(iii)$ in \thref{T:maintheorem1} and Fatou's lemma, we obtain
    \[ J := \sum_n \frac{1-|a_n|^2}{|1-a_n|^{2m+2}} + \int_\T \frac{d\nu(\zeta)}{|1 - \zeta|^{2m+2}} + \int_\T \frac{-\log|b(\zeta)|}{|1 - \zeta|^{2m+2}} d\m(\zeta) < \infty.\] 
     According to \eqref{E:EzEstimate}, for any $z \in \D$ we have the estimate
    \[ \frac{1}{|1-\conj{w}z|} \leq \frac{4}{|1-w|}, \quad w \in \cD, w \not\in E_z \cup S_z.\] Dominated Convergence Theorem thus implies
    \begin{equation} \label{E:NotEzSzIntegralConv}
        \lim_{\substack{z \to 1, z \in \Omega}} \sum_{n: a_n \not\in S_z} \frac{1-|a_n|^2}{|1-\conj{a_n}z|^{2m+2}} + \int_{\T \setminus E_z} \frac{d\nu(\zeta)}{|1 - \conj{\zeta}z|^{2m+2}} + \int_{\T \setminus E_z} \frac{-\log|b(\zeta)|}{|1 - \conj{\zeta}z|^{2m+2}} d\m(\zeta) = J.
    \end{equation} For $z = r \in (0,1)$, we have by definition that $E_z = S_z = \varnothing$, and so if the limit in \eqref{E:LimIntCond} exists, it follows from \eqref{E:NotEzSzIntegralConv} that it must equal $J$. Then the difference of the quantities in \eqref{E:LimIntCond} and \eqref{E:NotEzSzIntegralConv} must tend to $0$ as $z \to 1$ in $\Omega$. This is precisely condition \eqref{E:LimIntCondEz}.
    Conversely, assuming \eqref{E:LimIntCondEz}, it follows immediately from \eqref{E:NotEzSzIntegralConv} that \eqref{E:LimIntCond} exists (and equals $J$).    
\end{proof}

As in the previous section, we will proceed by dividing the rest of the proof into the two cases: Blaschke products, and zero-free functions.

\subsection{Blaschke product case.}

Consider the space $\ell^1$ of summable unilateral sequences $x = \big( x(n)\big)_n$ indexed by natural numbers $n \geq 1$ and denote by $\| x \|_1 := \sum_n |x(n)|$ the usual norm in $\ell^1$.  It will be useful to introduce the sequences $A_{k,l,z}$ given by
\[ A_{k,l, z}(n) = \frac{1-|a_n|^2}{(1-\conj{a_n}z)^{1+k}(1-a_n\conj{z})^{1+l}}, \quad n \geq 1.\] Here $a_n$ are the zeros of the Blaschke product $b$. By the Blaschke condition $\sum_n 1-|a_n| < \infty$, the sequences $A_{k,l,z}$ are members of $\ell^1$. The condition $(iii)$ of \thref{T:maintheorem1} on $b$ reduces to \eqref{E:PartIIIforBlaschke} and it implies that
\[ \sup_{z \in \Omega, 1 \leq k,l \leq m} \|A_{k,l,z}\|_1 < \infty\] and \[ A_{k,l,1} \in \ell^1, \quad 1 \leq k,l \leq m.\] Moreover, we have the following norm convergence result.

\begin{lem} \thlabel{L:ProofTheorem3BlaschkeNormConvLemma}
If the equivalent conditions of \thref{T:maintheorem1} are satisfied for the Blaschke product $b$ and integer $m$, then for $k+l < 2m$ we have
    \[\lim_{\substack{z \to 1 \\ z \in \Omega}} \| A_{k,l, z} - A_{k,l,1}\|_1 = 0.\]
\end{lem}

\begin{proof}
    Because for every $n$ it holds that $A_{k,l,z}(n) \to A_{k,l,1}(n)$ as $z \to 1$, it is a standard result in measure theory that the desired norm convergence will follow from
    \[ \lim_{\substack{z \to 1 \\ z \in \Omega}} \| A_{k,l, z} \|_1 = \| A_{k,l,1}\|_1.\] Now, this convergence is established by an argument completely analogous to the one used in the proof of \thref{L:TheoremBLemma3} if we can show that
    \begin{equation}
        \label{E:SufficientLimit}
        \lim_{\substack{z \to 1 \\ z \in \Omega}} \sum_{n : a_n \in S_z} \frac{1-|a_n|^2}{|1-\conj{a_n}z|^{2+k+l}} = 0.
    \end{equation}  To see this, we prove first that
    \begin{equation}
        \label{E:LimSupAnZestimate}
        \lim_{\substack{z \to 1 \\ z \in \Omega }} \sup_{n : a_n \in S_z} |1-\conj{a_n}z| = 0.
    \end{equation} Should this not be true, then there exists a positive number $\delta$, a sequence $(z_n)_n$ of points in $\Omega$ converging to $1$, and a sequence $(a^*_n)_n$ of distinct zeros of $b$, such that $a^*_n \in S_z$ and $\delta < |1-\conj{a^*_n} z_n|$ for all $n$. Now, the standard estimate
    \[|1-\conj{a^*_n}z_n| \leq (1-|a_n^*|) + (1-|z_n|) + |\arg (\conj{a^*_n}z_n)| \] shows that $1-|a_n^*|$ is bounded from below, since the other two terms on the right-hand side above tend to $0$ as $z \to 1$. This says that the Blaschke product has infinitely many distinct zeros in a disk centered at the origin of radius strictly less than $1$, which is impossible. This establishes \eqref{E:LimSupAnZestimate}. 
    Accordingly, given any $\epsilon > 0$, if $z \in \Omega$ is sufficiently close to $1$, then $\sup_{n : a_n \in S_z} |1-\conj{a_n}z| < \epsilon$, and hence
    \begin{align*}
        \limsup_{\substack{z \to 1 \\ z \in \Omega}} \sum_{n : a_n \in S_z} \frac{1-|a_n|^2}{|1-\conj{a_n}z|^{2+k+l}} &= \limsup_{\substack{z \to 1 \\ z \in \Omega }}  \sum_{n : a_n \in S_z} |1-\conj{a_n}z|^{2m-k-l}\frac{1-|a_n|^2}{|1-\conj{a_n}z|^{2m+2}} \\
        &\leq \epsilon^{2m-k-l} \sup_{z \in \Omega} \sum_n \frac{1-|a_n|^2}{|1-\conj{a_n}z|^{2m+2}}.
    \end{align*} Since $2m-k-l \geq 1$, we see from the above estimate that \eqref{E:SufficientLimit} holds, and so the proof is complete.
\end{proof}

We may now finish the proof of the equivalence of $(i)$ and $(ii)$ in \thref{T:maintheorem2} in the case $b$ is a Blaschke product. We go back to the equations \eqref{E:BlaschkeDerivativeKernelNormFormula} and \eqref{E:CklDef}, and note that \[ C_{k,l}(z) = \sum_j B_{k,l,z}(j) A_{k,l,z}(j)\] where by \thref{L:bjSupBoundLemma} the sequences
$$B_{k,l,z} = (B_{k,l,z}(j) \big)_j
=\left({m \choose k} {m \choose l} b^{(m-k)}_{j-1}(z) \conj{b^{(m-l)}_{j-1}(z)} \conj{a_j}^k a_j^l k!l!\right)_j
$$
are uniformly bounded for $z \in \Omega$ and converge in each coordinate as $z \to 1$ in $\Omega$. It follows from \thref{L:ProofTheorem3BlaschkeNormConvLemma} that for $k +l < 2m$, $C_{k,l}(z)$ converges as $z \to 1$ in $\Omega$. The remaining term on the right-hand side of \eqref{E:BlaschkeDerivativeKernelNormFormula} is
\[ C_{m,m}(z) = \sum_{j \geq 1} |b_{j-1}(z)|^2 |a_j|^{2m} (m!)^2 \frac{1-|a_j|^2}{|1-\conj{a_j}z|^{2m+2}}\] and it follows that this term converges if and only if $\partial_z^m \conj{\partial_z}^m \|k^b_z\|^2_b$ converges. By the inequality $|b(z)| \leq |b_{j-1}(z)|$ for all $z \in \D$ and the convergence $|b(z)| \to 1$ as $z \to 1$ in $\Omega$ (recall \thref{L:bjSupBoundLemma}), as well as the convergence $|a_j| \to 1$ as $j \to \infty$, we see that $C_{m,m}(z)$ converges as $z \to 1$ in $\Omega$ if and only if $\sum_j \frac{1-|a_j|^2}{|1-\conj{a_j}z|^{2m+2}}$ converges as $z \to 1$ in $\Omega$. Our proof of \thref{T:maintheorem2} in the Blaschke product case is thus complete after an invocation of \thref{L:TheoremBLemma1} and \thref{L:TheoremBLemma3}.

\subsection{Zero-free case.}
In order to prove \thref{T:maintheorem2} for $b$ zero-free, recall first the equation \eqref{E:muSKbExpressionLEibniz}, which we repeat here for convenience of the reader:
\begin{equation}
    \label{E:muSKbExpressionLEibnizRepeated}
    \int_\T \frac{d\mu(\zeta)}{|1-\conj{\zeta}z|^{2m+2}} = \sum_{i,j=0}^m P_{i,j}(z) \cdot \partial^i_z \conj{\partial_z}^j \|k^b_z\|_b^2.
\end{equation}
We note that 
\[ \lim_{\substack{z \to 1 \\ z \in \Omega}} P_{i,j}(z)\] exist by \thref{L:bZeroFreeSupBoundLemma} and the assumption of validity of the conditions in \thref{T:maintheorem1}, since $P_{i,j}(z)$ are products of derivatives of $b(z)$ and $S(|b(z)|^2)$ of order $m$ and below.

Assume now that the first condition in \thref{T:maintheorem2} is satisfied, i.e., we have the norm convergence $\conj{\partial_z}^m k^b_{z} \to \conj{\partial_z}^m k^b_1$ as $z \to 1$ in $\Omega$. By \thref{L:TheoremBLemma2}, for every pair of non-negative integers $i,j \leq m$, we have that
\begin{equation}
    \label{E:DerivKernelIPconv}
    \lim_{\substack{z \to 1 \\ z \in \Omega}} \partial^i_z \conj{\partial_z}^j \|k^b_z\|^2_b = \lim_{\substack{z \to 1 \\ z \in \Omega}} \ip{\conj{\partial_z}^i b_z}{\conj{\partial_z}^j k^b_z}_b
\end{equation} exist. Thus the right-hand side in \eqref{E:muSKbExpressionLEibnizRepeated} converges as $z \to 1$ in $\Omega$. Hence so does the left-hand side. By \thref{L:TheoremBLemma3}, we see that part $(ii)$ of \thref{T:maintheorem2} holds for zero-free $b$.

Now let us prove the converse implication $(ii) \Rightarrow (i)$. Proof proceeds by induction on $m$. Validity of the base case $m=0$ follows immediately from \eqref{E:muSKbExpression}, \thref{L:TheoremBLemma1} and \thref{L:TheoremBLemma3}, since the term $S(|b(z)|^2)$ converges as $z \to 1$ in $\Omega$. For $m > 0$, we go back to the formula \eqref{E:muSKbExpressionLEibnizRepeated}. It is not hard to see that condition $(ii)$ of \thref{T:maintheorem2} for some $m$ implies the same statement for $m - 1$. Then by the induction hypothesis we conclude that the limits in \eqref{E:DerivKernelIPconv} exists for $i,j \leq m-1$. For $i = m$ and $j < m$, the limit in \eqref{E:DerivKernelIPconv} exists by the weak convergence of $\conj{\partial_z}^m k^b_z$ and norm convergence of $\conj{\partial_z}^jk^b_z$. Thus we have shown that all but one term on both sides of the equation in \eqref{E:muSKbExpressionLEibnizRepeated} converge as $z \to 1$ in $\Omega$, the exceptional term being indexed by $i,j = m$. This term equals $S(|b(z)|^2) \partial^m_z \conj{\partial_z}^m \|k^b_z\|^2_b$, and it must of course also converge. Then so does $\partial^m_z \conj{\partial_z}^m \|k^b_z\|^2_b$, and the proof of the implication $(ii) \Rightarrow (i)$ of \thref{T:maintheorem2} is complete in the case of zero-free $b$.

\subsection{Combining the special cases into the general result.}

Consider, as before $b = B_1B_2$, where $B_1$ is a Blaschke product and $B_2$ is zero-free. We have just proved \thref{T:maintheorem2} for $B_1$ and $B_2$.

Assume that part $(i)$ of \thref{T:maintheorem1} holds for $b$. Arguing in the same way as we did in \eqref{E:FunctionalNormSubspaceEstimate} but for differences of kernels, we obtain that
\begin{equation}
    \label{E:KernelBiNormConv}
    \lim_{\substack{z, z' \to 1 \\ z \in \Omega}} \|\conj{\partial_z}^m k^{B_i}_z - \conj{\partial_z}^m k^{B_i}_{z'}\|_{B_i} = 0, \quad i = 1,2.
\end{equation} Hence $\conj{\partial_z}^m k^{B_i}_z$ converges in norm as $z \to 1$ in $\Omega$. By the special cases of \thref{T:maintheorem2} proved for $B_1$ and $B_2$, we see that part $(ii)$ in \thref{T:maintheorem2} holds for $b = B_1B_2$.

Conversely, assume that part $(ii)$ of \thref{T:maintheorem2} holds for $b = B_1B_2$. Then it holds for $B_1$ and $B_2$ separately, and so \eqref{E:KernelBiNormConv} holds. Applying the operator $\conj{\partial_z}^m$ to \[ k^b_z(w) = k^{B_1}_z(w) + B_1(w)\conj{B_1(z)}k^{B_2}_z(w)\] results in
\[ \conj{\partial_z}^m k^b_z(w) = \conj{\partial_z}^m k^{B_1}_z(w) + B_1(w)\sum_{i=0}^m {m \choose i} \conj{\partial_z}^{m-i} \conj{B_1(z)} \conj{\partial_z}^i k^{B_2}_z(w).\] It is known that $B_1$ is a bounded multiplier from $\h(B_2)$ into $\hb$ (see \cite[page 5]{sarasonbook}). By \thref{L:bjSupBoundLemma}, the coefficients $\conj{\partial_z}^{m-i} \conj{B_1(z)}$ converge as $z \to 1$ in $\Omega$. By the contractive containments $\h(B_i) \subset \hb$, \thref{L:TheoremBLemma2} and \eqref{E:KernelBiNormConv}, the terms $\conj{\partial_z}^m k^{B_1}_z$ and $\conj{\partial_z}^i k^{B_2}_z$ converge in norm in $\hb$ as $z \to 1$ in $\Omega$. Thus $\conj{\partial_z}^m k^b_z$ converges in norm as $z \to 1$ in $\Omega$, and the proof is complete.

\section{Proof of Theorems C and D}
\label{S:GeneralizationFailureSec}

In the proof we will use the following simple calculus lemma.

\begin{lem} \thlabel{L:calcLemma}
    Let $\rho$ be an increasing continuous function on $[0, \infty)$, continuously differentiable on $(0,\infty)$. Assume that
    \[ \lim_{x \to 0} \rho'(x) = 0. \]
    For all sufficiently small positive $x$ and $x^*$ satisfying $x < x^*$, we have
    \[ 2|e^{ix^*} - z| \geq \rho(x^*), \quad z = (1-\rho(x))e^{ix}.\]
\end{lem}

\begin{proof}
    We have \[2|e^{ix^*} - z| \geq \frac{1}{2}|x^* - x| + \rho(x) := D(x).\] Note that for $x < x^*$ we have $D'(x) = -\frac{1}{2} + \rho'(x)$ which, if $x$ is sufficiently small, is a negative quantity by the assumption on $\rho$. Thus $D$ is decreasing on the interval $(x,x^*)$, and therefore $\rho(x^*) = D(x^*) \leq D(x) \leq 2|e^{ix^*} - z|$.
\end{proof}

\begin{proof}[Proof of \thref{T:maintheorem3}]
    Since $\rho(x) = \int^x_0 \rho'(x) \, dx = o(x)$, by our assumption we can pick a decreasing sequence $(x_n)_n$ of positive real numbers which converges to $0$ and satisfies
    the two conditions
    \begin{equation}
        \label{E:RhoXnTwoCond1}
     \sum_n \frac{\rho^2(x_n)}{x^2_n} < \infty
    \end{equation}
    and
    \begin{equation}
        \label{E:RhoXnTwoCond2}
    x_{n+1} < x_n/2.
    \end{equation}
    We take $b$ to be the singular inner function
    \[ b(z) = \exp \Big(-\int_\T \frac{\zeta + z}{\zeta - z} d\nu(\zeta) \Big), \quad z \in \D\] where 
    \[ \nu = \sum_n \rho^2(x_n) \delta_{e^{ix_n} }\] and $\delta_{e^{ix_n}}$ is a point mass at $e^{ix_n} \in \T$. We will soon show that $\sup_{z \in \Omega_\rho} \|k^b_z\|_{\hb} < \infty$ holds. Assuming this, according to \thref{T:maintheorem1} (or the Duan-Li-Mashreghi Theorem) and Fatou's lemma, we have
    \[ \int_\T \frac{d\nu(\zeta)}{|1-\zeta|^2} < \infty.\] Moreover, $\lim_{\substack{z \to 1 \\ z \in \Omega_\rho}} b(z)$ exists. Noting that \[ \frac{b'(z)}{b(z)} = -\int_{\T} \frac{2 \zeta d\nu(\zeta)}{(\zeta-z)^2},\] to show that $\lim_{z \to 1, z \in \Omega_\rho} b'(z)$ does not exist, it will suffice to show that 
    \[ \lim_{\substack{z \to 1 \\ z \in \Omega_\rho}} \int_{\T} \frac{2\zeta d\nu(\zeta)}{(\zeta - z)^2} \] does not exist. To see this, we will check that the radial limit $z = r \to 1$ and the limit along the sequence $z_n = (1-\rho(x_n))e^{ix_n} \in \partial \Omega_\rho$ differ. For the case of the radial limit, we employ the inequality $\frac{1}{|\zeta - r|^2} \leq \frac{4}{|\zeta - 1|^2}$ and the Dominated Convergence Theorem to see that
    \[ \lim_{r \to 1} \int_{\T} \frac{2\zeta d\nu(\zeta)}{(\zeta - r)^2}  = \int_{\T} \frac{2\zeta d\nu(\zeta)}{(\zeta - 1)^2}.\] Recall the definition of the arcs $E_{z_n}$ in \eqref{E:EzDef}. This time, the Dominated Convergence Theorem and the estimate in \eqref{E:EzEstimate} show that
    \[ \lim_{n \to \infty} \int_{\T \setminus E_{z_n}} \frac{2\zeta d\nu(\zeta)}{(\zeta - z_n)^2}  = \int_{\T} \frac{2\zeta d\nu(\zeta)}{(\zeta - 1)^2},\] same as the radial limit. However, according to $(\ref{E:RhoXnTwoCond2})$, for $z_n := (1-\rho(x_n))e^{ix_n} \in \partial \Omega_\rho$, we have that $E_{z_n} \cap \supp \nu = \{e^{ix_n} \}$, and so 
    \[ \lim_{n \to \infty} \int_{ E_{z_n}} \frac{2\zeta d\nu(\zeta)}{(\zeta - z_n)^2}  = \lim_{n \to \infty} 2e^{-ix_n} = 2.\] Thus 
    \[ \lim_{n \to \infty} \int_{\T} \frac{2\zeta d\nu(\zeta)}{(\zeta - z_n)^2}  =  \lim_{r \to 1} \int_{\T} \frac{2\zeta d\nu(\zeta)}{(\zeta - r)^2} + 2,\] and it follows that     
    $\lim_{\substack{z \to 1 \\ z \in \Omega_\rho}} b'(z)$ does not exist. 
    
    It remains to show that $\sup_{z \in \Omega_\rho} \|k^b_z\|_b  < \infty$. According to \thref{T:maintheorem1}, we need to show the finiteness of
    \begin{equation}
        \label{E:SupQuantityNorm}
        \sup_{z \in \Omega_\rho} \int_\T \frac{d\nu(\zeta)}{|z-\zeta|^2} = \sup_{z \in \Omega_\rho} \sum_n \frac{\rho^2(x_n)}{|e^{ix_n} - z|^2}. 
    \end{equation} It will suffice to do this with the above suprema restricted to $z \in \partial \Omega_\rho$. Moreover, it will suffice to do this for all $z \in \Omega_\rho$ sufficiently close to $1$.

    Let then $z  = (1-\rho(x))e^{ix} \in \partial \Omega_\rho$. There exists a unique $m$ for which we have $x_{m+1} \leq x < x_m$. We will divide the sum $\sum_n \frac{\rho^2(x_n)}{|e^{ix_n} - z|^2}$ into terms index by $n \geq m+2$ and $n \leq m-1$, with special arguments for terms indexed by $n = m+1$ and $n=m$. 
    
    Let us treat the two special terms first. We will show that $|e^{ix_m} - z|$ and $|e^{ix_{m+1}}-z|$ are controlled from below by $\rho(x_m)$ and $\rho(x_{m+1})$ respectively. For $n=m+1$, note that since $\rho$ is increasing, $1-|z| = \rho(x) \geq \rho(x_{m+1})$ and so $|e^{ix_{m+1}}-z| \geq 1-|z| \geq \rho(x_{m+1})$. For $n = m$, apply \thref{L:calcLemma} to $x^* = x_m$ to conclude that $2|e^{ix_{m}} - z| \geq \rho(x_m)$. It follows from these two estimates that
    \[ \frac{\rho^2(x_m)}{|e^{ix_{m}} - z|^2} + \frac{\rho^2(x_{m+1})}{|e^{ix_{m+1}} - z|^2} \leq 5.\]

    Consider now the terms indexed by $n \geq m+2$. Then $x = \arg z$ lies further away from $x_n$ than $x_{m+1}$ does, and it follows that $|x_{m+1} - x_n| \leq |x_n - x| \leq 4|e^{ix_n} - z|$. Moreover, by \eqref{E:RhoXnTwoCond2} we have $2x_n < x_{m+1}$ which then implies that $x_n \leq |x_{m+1}- x_{n}| \leq 4|e^{ix_n} - z|$. We conclude that
    \[ \sum_{n \geq m+2} \frac{\rho^2(x_n)}{|e^{ix_{n}}-z|^2} \leq 16\sum_{n : n \geq m+2} \frac{\rho^2(x_n)}{x_n^2}.\] So this part of the sum is controlled by \eqref{E:RhoXnTwoCond1}.

    At last, consider the terms indexed by $n \leq m-1$. Then, similarly to the previous part, we get $4|z-e^{ix_{n}}| > |x_m - x_n| > x_n/2$ and so
    \[ \sum_{n : n \leq m-1} \frac{\rho^2(x_n)}{|e^{ix_{n}}-z|^2} \leq 64 \sum_{n : n \leq m-1} \frac{\rho^2(x_n)}{x_n^2}.\] 
    
    Our above estimates combine to show that $\sup_{z \in \Omega_\rho} \|k^b_z\|_b < \infty$.  
\end{proof}

Finally, let us now prove \thref{T:maintheorem4}. Take the logarithmic derivative of $b$ and combine the earlier-derived expressions \eqref{E:BlaschkeLogarithmicDerivative} and \eqref{E:LogarithmicDerivativeZeroFree} to obtain
\begin{equation}
    \label{E:LogarithmicDerivativeFull}
    \frac{b'(z)}{b(z)} = \sum_{n}\frac{1-|a_n|^2}{(z-a_n)(1-\conj{a_n}z)} -\int_\T \frac{2\zeta d\mu(\zeta)}{(\zeta - z)^2}.
\end{equation}
with $d\mu = -d\nu + \log |b| d\m$, as before. According to the standard measure theory result already mentioned at the beginning of the proof of \thref{L:ProofTheorem3BlaschkeNormConvLemma}, the right-hand side above will converge to $\sum_{n} \frac{1-|a_n|^2}{|1-\conj{a_n}|^2} - \int_{\T} \frac{2 \zeta d\mu(\zeta)}{(\zeta - 1)^2}$ as $z \to 1$ in $\Omega$ if we can show that
\[ \lim_{\substack{z \to 1 \\ z \in \Omega}} \sum_{n}\frac{1-|a_n|^2}{|z-a_n||1-\conj{a_n}z|} = \sum_{n}\frac{1-|a_n|^2}{|1-a_n|^2} \] 
and
\[ \lim_{\substack{z \to 1 \\ z \in \Omega}} \int_\T \frac{ d\mu(\zeta)}{|\zeta - z|^2} = \int_\T \frac{ d\mu(\zeta)}{|\zeta - 1|^2}\] Since 
\[ 1 = \lim_{\substack{z \to 1 \\ z \in \Omega}} |b(z)| \leq \limsup_{\substack{z \to 1 \\ z \in \Omega}} \frac{|z-a_n|}{|1-\conj{a_n}z|} \leq 1,\] it follows that 
\[ \lim_{\substack{z \to 1 \\ z \in \Omega}} \sum_{n}\frac{1-|a_n|^2}{|z-a_n||1-\conj{a_n}z|} = \lim_{\substack{z \to 1 \\ z \in \Omega}} \sum_{n}\frac{1-|a_n|^2}{|1-\conj{a_n}z|^2},\] if these limits exist. Now we may simply repeat the argument given in the proof of \thref{L:TheoremBLemma3} to establish the required convergence of the right-hand side above. Hence the right-hand side in \eqref{E:LogarithmicDerivativeFull} converges as $z \to 1$ in $\Omega$. Since the denominator $b(z)$ on the left-hand side also converges, we obtain the desired result.

\bibliographystyle{alpha}
\bibliography{mybib}

\end{document}